\documentclass{amsart}
\usepackage{graphicx}
\newtheorem{theorem}{Theorem}[section]
\newtheorem{lemma}[theorem]{Lemma}
\newtheorem{claim}[theorem]{Claim}
\newtheorem{problem}[theorem]{Problem}
\newtheorem{example}[theorem]{Example}

\theoremstyle{definition}
\newtheorem{definition}[theorem]{Definition}

\newcommand{\Z}{\ensuremath{\mathbb Z}}

\begin{document}

\title[An obstruction to embedding $2$-dimensional complexes]{An obstruction to embedding $2$-dimensional complexes into the $3$-sphere}

\author{Kazufumi Eto}
\address{Department of Mathematics, Nippon Institute of Technology, Saitama, 345-8501, Japan}
\email{etou@nit.ac.jp}

\author{Shosaku Matsuzaki}
\address{Integrated Arts and Sciences, Waseda University, 1-6-1 Nishiwaseda, Shinjuku-ku, Tokyo, 169-8050, Japan}
\email{shosaku@aoni.waseda.jp}

\author{Makoto Ozawa}
\address{Department of Natural Sciences, Faculty of Arts and Sciences, Komazawa University, 1-23-1 Komazawa, Setagaya-ku, Tokyo, 154-8525, Japan}
\email{w3c@komazawa-u.ac.jp}
\thanks{The third author is partially supported by Grant-in-Aid for Scientific Research (C) (No. 26400097), The Ministry of Education, Culture, Sports, Science and Technology, Japan}



\begin{abstract}
We consider an embedding of a $2$-dimensional CW complex into the $3$-sphere, and construct it's dual graph.
Then we obtain a homogeneous system of linear equations from the $2$-dimensional CW complex in the first homology group of the complement of the dual graph.
By checking that the homogeneous system of linear equations does not
have an integral solution, we show that some $2$-dimensional CW complexes cannot be embedded into the 3-sphere.
\end{abstract}

\maketitle

\section{Introduction}

It is a fundamental problem to determine whether or not there exists an embedding from a topological space into another one.
The Menger--N\"{o}beling theorem ({\cite[Theorem 1.11.4.]{E}}) shows that any finite $n$-dimensional CW complex can be embedded into the $2n+1$-dimensional Euclidian space $\mathbb{R}^{2n+1}$.
Kuratowski (\cite{K}) proved that a $1$-dimensional complex $G$ cannot be embedded into $\mathbb{R}^2$ if and only if $G$ contains the complete graph $K_5$ or the bipartite graph $K_{3,3}$ as a subspace.
One can consider two extensions of Kuratowski's theorem, that is, to determine whether or not there exists an embedding from a $2$-dimensional complex $X$ into $\mathbb{R}^3$ or $\mathbb{R}^4$.
In this paper, we consider the former case, namely, the embeddings of $2$-dimensional complexes into $\mathbb{R}^3$, and give a necessary condition for a $2$-dimensional complex, whose $1$-skeleton is a closed $1$-manifold, to be embeddable into $\mathbb{R}^3$.
It is shown in \cite{MSTW} that the following algorithmic problem is decidable: given a $2$-dimensional simplicial complex, can it be embedded in $\mathbb{R}^3$?
As a remark on the latter case, in general, it is known that the union of all $n$-faces of a $(2n+2)$-simplex cannot be embedded in $\mathbb{R}^{2n}$ for any natural number $n$ ({\cite[1.11.F]{E})}.

Throughout this paper, we will work in the piecewise linear category.
As a matter of convenience, we introduce the following complexes.

\begin{definition}[Multibranched surface]
Let $S_n$ denote the quotient space obtained from a disjoint union of $n$ copies of $\mathbb{R}^2_+=\{(x,y)|y\ge 0\}\subset \mathbb{R}^2$ by gluing together along their boundary $\partial \mathbb{R}^2_+=\{(x,y)|y=0\}$ for each positive integer $n$.



A second countable Hausdorff space $X$ is called a {\em multibranched surface} if $X$ contains a disjoint union of simple closed curves $l_1,\ldots, l_n$, each of which we call a {\em branch}, satisfying the following:
\begin{itemize}
\item For each point $x \in l_1\cup\cdots\cup l_n$, there exists an open neighborhood $U$ of $x$ and a positive integer $i$ such that $U$ is homeomorphic to $S_i$.
\item For each point $x\in X-(l_1\cup\cdots\cup l_n)$, there exists an open neighborhood $U$ of $x$ such that $U$ is homeomorphic to $\mathbb{R}^2$.
\end{itemize}


We call each component, which is denoted by $e_i$, of $X-(l_1\cup\cdots\cup l_n)$ a {\em sector}.
\end{definition}

Hereafter, we assume that a multibranched surface $X$ is compact and all sectors of $X$ are orientable.

\begin{definition}[Algebraic degree]
Then we define the {\em algebraic degree} $ad_{e_i}(l_j)$ of an oriented sector $e_i$ on an oriented branch $l_j$ as 
\[
ad_{e_i}(l_j)[l_j]:=(f_i)_* \Big( \big[ \partial M_i \cap f_i^{-1}(l_j) \big] \Big)
\]
in $H_1(l_j ; \mathbb{Z})$, where $M_i$ is the closure of $e_i$ and $(f_i)_*$ is the induced homomorphism of an inclusion map $f_i: M_i \to X$.
See Figure \ref{algebraic_degree} for example.
\begin{figure}[htbp]
	\begin{center}
	\includegraphics[trim=0mm 0mm 0mm 0mm, width=.4\linewidth]{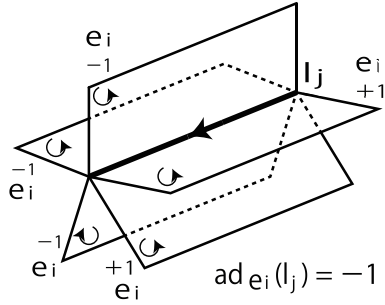}
	\end{center}
	\caption{Counting the algebraic degree $ad_{e_i}(l_j)$}
	\label{algebraic_degree}
\end{figure}
\end{definition}

If a multibranched surface was embedded into the $3$-sphere, then there is a possibility of several embeddings near branches for sectors depending on their circular permutations. See Figure \ref{circular_permutation} for example.

\begin{figure}[htbp]
	\begin{center}
	\includegraphics[trim=0mm 0mm 0mm 0mm, width=.6\linewidth]{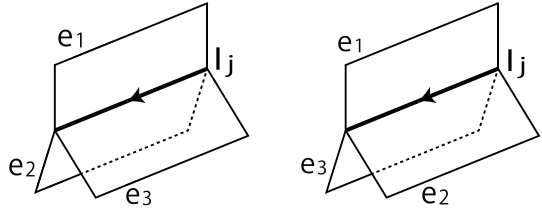}
	\end{center}
	\caption{Circular permutations of sectors}
	\label{circular_permutation}
\end{figure}

\begin{definition}[Abstract dual graph]
Suppose that we have fixed a circular permutation of sectors on each branch.
Then we can construct a directed graph as follows.
First we take a product $e_i \times [-1, 1]$ of each oriented sector $e_i$ and put $e_i^{\pm}=e_i \times \{ \pm 1\}$.
Next we glue $e_i^{\pm}$ along their boundaries depending on the circular permutation of sectors on each branch, and obtain closed surfaces denoted by $R_1, \ldots, R_t$.
Finally we construct the {\em abstract dual graph} whose vertex $v_j$ corresponds to a closed surface $R_j$, and there is a directed edge $e_i^*$ from $v_k$ to $v_j$ if and only if  $e_i^{+} \subset R_j$ and $e_i^- \subset R_k$.
See Figure \ref{dual_graph}.
\begin{figure}[htbp]
	\begin{center}
	\begin{tabular}{cc}
	\includegraphics[trim=0mm 0mm 0mm 0mm, width=.5\linewidth]{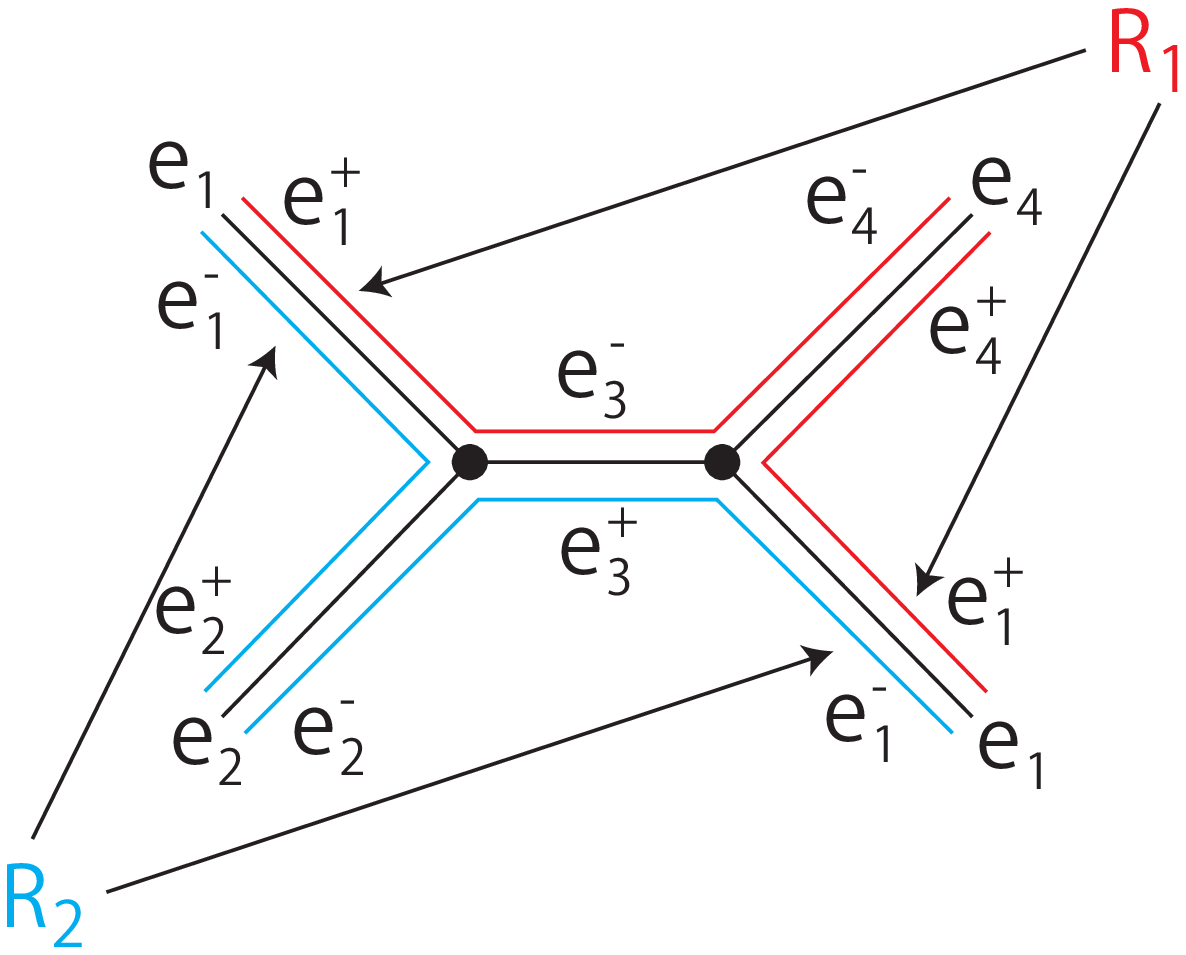}&
	\includegraphics[trim=0mm 0mm 0mm 0mm, width=.25\linewidth]{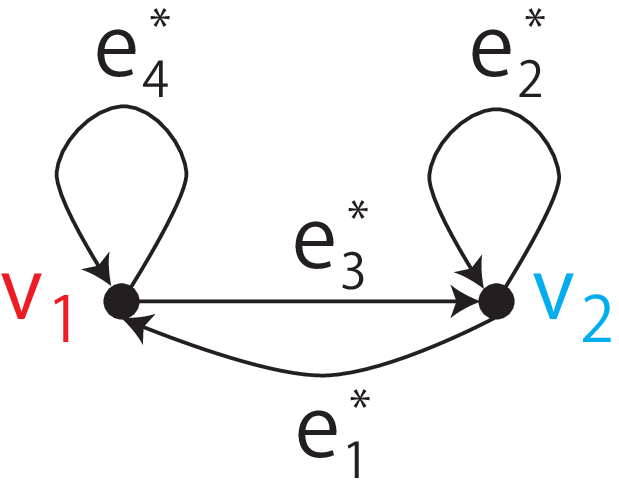}\\
	$X$ & $G_X$
	\end{tabular}
	\caption{An abstract dual graph $G_X$ of a multibranched surface $X$}
	\label{dual_graph}
	\end{center}
\end{figure}
\end{definition}

A multibranched surface $X$ is said to be {\em orientable} if all closed surfaces $R_1, \ldots, R_t$ are orientable.

\begin{definition}[Algebraic degree matrix]
Let $G_X$ be an abstract dual graph of a multibranched surface $X$.
Take a spanning tree $T$ of $G_X$.
Denote the sectors by $e_1,\ldots,e_m$ which correspond to edges $e_1^*, \ldots, e_m^*$ not belong to $T$.
We define the $m\times n$ {\em algebraic degree matrix} $A_T=(a_{ij})$ by 
\[
a_{ij} := ad_{e_i}(l_j),
\]
where $m$ is the first Betti number of $G_X$ and $n$ is the number of branch of $X$.
\end{definition}

\begin{example}\label{projective_plane}
We regard the real projective plane as a multibranched surface $X=e^0\cup e^1\cup e^2$ as shown in Figure \ref{real_projective_plane}.
Then the algebraic degree of $e^2$ on the branch $l=e^0\cup e^1$ is $ad_{e^2}(l)=2$, and the abstract dual graph $G_X$ of $X$ is a bouquet with a vertex $v$ corresponding to $S^2=e^2_+\cup e^2_-$ and a loop $(e^2)^*$ corresponding to $e^2$.
The spanning tree $T$ is a single vertex $v$, and the algebraic degree matrix is $A_T=[2]$.
\begin{figure}[htbp]
	\begin{center}
	\begin{tabular}{ccc}
	\includegraphics[trim=0mm 0mm 0mm 0mm, width=.25\linewidth]{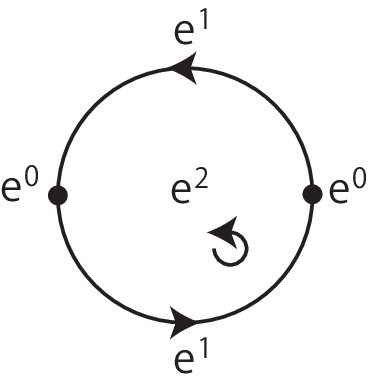}&
	\includegraphics[trim=0mm 0mm 0mm 0mm, width=.45\linewidth]{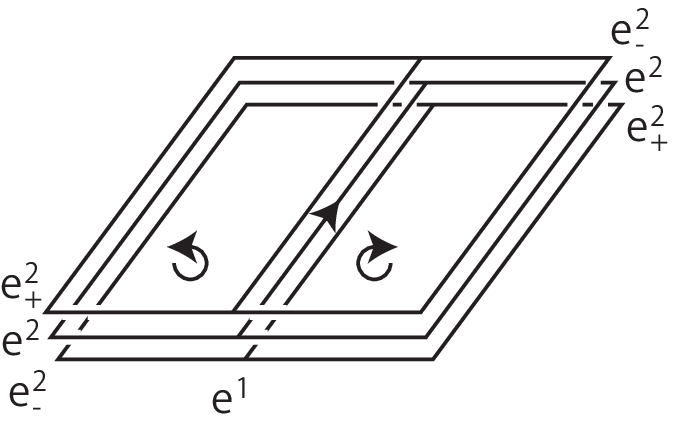}&
	\includegraphics[trim=0mm 0mm 0mm 0mm, width=.2\linewidth]{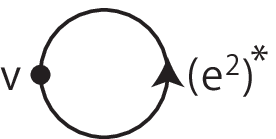}\\
	$X$ & $X\times I$ & $G_X$
	\end{tabular}
	\caption{The real projective plane as a multibranched surface, and its abstract dual graph}
	\label{real_projective_plane}
	\end{center}
\end{figure}
\end{example}

The following is a main theorem in this paper.
By using this theorem, we will give critical examples in Section 3.

\begin{theorem}\label{main}
Let $X$ be a connected, compact, orientable, multibranched surface.
If for each abstract dual graph $G_X$ of $X$, one of the following conditions (1) and (2) holds, then $X$ cannot be embedded into the $3$-sphere $S^3$, where $m$ is the first Betti number of $G_X$ and $n$ is the number of branch of $X$.
\begin{enumerate}
\item $m>n$.
\item $m\le n$ and there exists a spanning tree $T$ of $G_X$ such that the greatest common divisor of all $m$-minor determinants of the algebraic degree matrix $A_T$.
\end{enumerate}
\end{theorem}

Theorem \ref{main} still holds even if we replace the 3-sphere with homology 3-spheres.
All arguments used in the proof of Theorem \ref{main} still work on homology 3-spheres since we have used only homological conditions of the 3-sphere.

\section{Proof of Theorem \ref{main}}

\begin{lemma}\label{minor}
Let $m, n>0$
and
$f : \Z^n\to\Z^m$ be a homomorphism
presented by the $m\times n$ matrix $A$.
Then the following are equivalent:
\begin{enumerate}
\item
there is a homomorphism  $g : \Z^m\to\Z^n$ such that the map $fg$
is the identity,
\item
$m\leq n$ and
the general common divisor of the $m$ minors of $A$ is one.
\end{enumerate}
\end{lemma}

\begin{proof}
$(1) \Rightarrow (2)$\quad
By the assumption, $f$ is injective and we have $m\leq n$.
Since the kernel $K$ of $g$ is a submodule of $\Z^n$,
it is free and
there is an isomorphism from $K\oplus\Z^m$ to $\Z^n$
sending $(a, b)\in K\oplus\Z^m$ to $a+f(b)$.
Hence the matrix presenting this isomorphism
contains $A$ as a submatrix
and its determinant is $1$ or $-1$.
It follows that 
the general common divisor of the $m$ minors of $A$ must be one.

$(2) \Rightarrow (1)$\quad
Let $D_1, \dots, D_t$ be the $m$ minors 
of an $m\times n$ matrix $(X_{ij})$ whose entries are indeterminants.
Then we may regard them as maps sending an $m\times n$ matrix
$B$
to an integer $D_l(B)$ for each $l$.
By the assumption,
there are integers $c_1, \dots, c_t$
satisfying $c_1D_1(A)+\cdots+c_tD_t(A)=1$.
Put $\varphi=c_1D_1+\cdots+c_tD_t$ be a map from $\Z^{m\times n}$
to $\Z$
and, for each $k, l$, 
$B_{kl}=(b_{ij})$ an $m\times n$ matrix
where $b_{kl}=1$, $b_{il}=0$ if $i\ne k$ 
and $b_{ij}=a_{ij}$ if $j\ne l$, where $A=(a_{ij})$.
And put $C$
be the transpose of an $m\times n$ matrix $(\varphi(B_{kl}))$.
Then $CA=1$
and $C$ presents an inverse map $g$ of $f$.
\end{proof}

\begin{proof}[Proof of Theorem \ref{main}]
Suppose that a connected, compact, orientable, multibranched surface $X$ can be embedded into the $3$-sphere $S^3$.
Then by regarding each component of $S^3-X$ as a vertex, and joining two vertices if the corresponding components are adjacent by a sector, we obtain a geometric dual graph for $X$ in $S^3$.
We note that each complementary region of $S^3-X$ has a connected boundary, and that the geometric dual graph is isomorphic to an abstract dual graph $G_X$ of $X$.

Take a spanning tree $T$ of $G_X$, and let $e_1,\ldots,e_m$ be the sectors which correspond to edges $e_1^*, \ldots, e_m^*$ not belong to $T$.
Take a meridian $m_i\subset e_i$ of $e_i^*$, and let $l_1,\ldots, l_n$ be the branches of $X$.

Since each $M_i$ is orientable, it holds that $[\partial M_i]+[f_i^{-1}(m_i)]=0$ in $H_1\big( M_i-f_i^{-1}(G_X \cap X) \big)$ and hence $(f_i)_* \big( [\partial M_i]+[f_i^{-1}(m_i)] \big) =0$ in $H_1(S^3-G_X; \mathbb{Z})$.
Therefore, we have
\[
\sum_{j=1}^n ad_{e_i}(l_j) [l_j] + [m_i]=0
\]
for $i=1,\ldots , m$.
This can be represented by the following matrix:
\[
A_T
\left[
	\begin{array}{c}
	\lbrack l_1 \rbrack \\
	\vdots \\
	\lbrack l_n \rbrack
	\end{array}
\right]
+
\left[
	\begin{array}{c}
	[m_1] \\
	\vdots \\
	\lbrack m_m \rbrack
	\end{array}
\right]
=
\left[
	\begin{array}{c}
	0 \\
	\vdots \\
	0 
	\end{array}
\right]
\]
Since $\{[m_1], \ldots , [m_m]\}$ is a generator for $H_1(S^3-G_X; \mathbb{Z})$, there exists an $n\times m$ matrix $B$ such that:
\[
A_TB
\left[
	\begin{array}{c}
	\lbrack m_1 \rbrack \\
	\vdots \\
	\lbrack m_n \rbrack
	\end{array}
\right]
+
\left[
	\begin{array}{c}
	\lbrack m_1 \rbrack \\
	\vdots \\
	\lbrack m_m \rbrack 
	\end{array}
\right]
=
\left[
	\begin{array}{c}
	0 \\
	\vdots \\
	0 
	\end{array}
\right]
\]
Thus, we have $A_T B=-E$ and by Lemma \ref{minor}, 
$m\leq n$ and
the general common divisor of the $m$ minors of $A_T$ is one.
\end{proof}

\section{Critical multibranched surfaces}

We say that a multibranched surface $X$ is {\em critical} if $X$ cannot be embedded into $S^3$ and for any $x\in X$, $X-x$ can be embedded into $S^3$.
We note that any multibranched surface can be embedded into the $4$-sphere $S^4$.
It is well-known that a multibranched surface which is homeomorphic to a non-orientable closed surface is critical.
This can be seen also by Example \ref{projective_plane} and Theorem \ref{main}
in the case for the real projective plane.

\begin{theorem}\label{X_1}
Let $X_1$ be a multibranched surface with a single sector $e_1$ which is homeomorphic to a planar surface, and with a single branch $l_1$.
Then $X_1$ cannot be embedded into $S^3$ if and only if $|ad_{e_1}(l_1)|\ge 2$.
Moreover, if $X_1$ cannot be embedded into $S^3$, then it is critical.
\end{theorem}

\begin{proof}
First we observe that any abstract dual graph of $X_1$ is a bouquet.
Since $det(A_T)= ad_{e_1}(l_1)$, by Theorem \ref{main}, $X_1$ cannot be embedded into the 3-sphere if $|ad_{e_1}(l_1)|\ge 2$.
Conversely, if $|ad_{e_1}(l_1)|<2$, then we can construct an embedding of $X_1$ into the 3-sphere.

Second we consider $X_1$ whose one point was removed.
We glue a pair of branches with different orientations on the same side alternatively, and if a pair of branches with a same orientation remained, then we glue them on the different sides by going a long way round as shown in Figure \ref{criticality_X_1}.
\end{proof}

\begin{figure}[htbp]
	\begin{center}
	\includegraphics[trim=0mm 0mm 0mm 0mm, width=.8\linewidth]{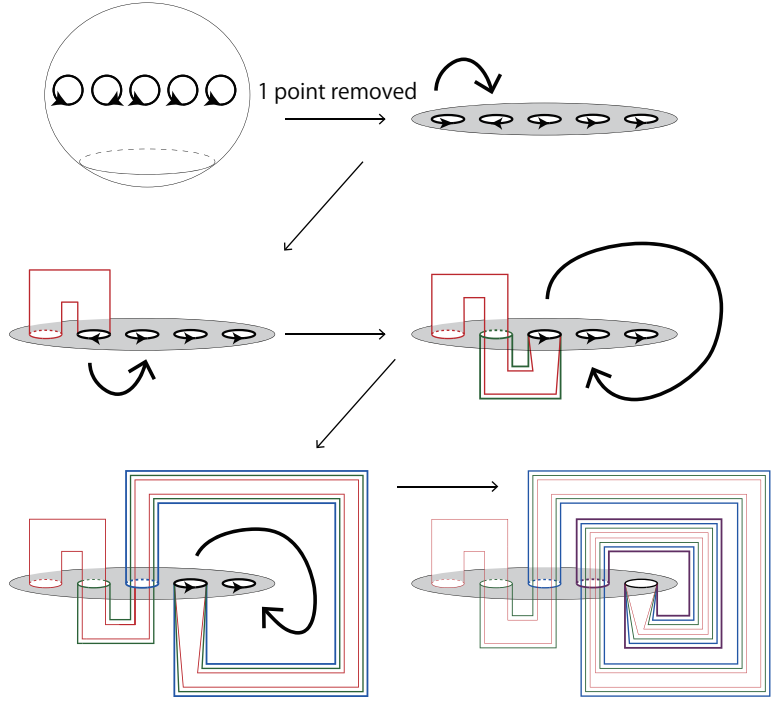}
	\end{center}
	\caption{Criticality of $X_1$}
	\label{criticality_X_1}
\end{figure}

\begin{theorem}\label{X_2}
Let $X_2$ be a multibranched surface obtained from sectors $e_1,\ldots, e_n$ by gluing along branches $l_1,\ldots, l_n$ as shown in Figure \ref{cells_X_2}.
Then $X_2$ is critical.
\end{theorem}

\begin{figure}[htbp]
	\begin{center}
	\includegraphics[trim=0mm 0mm 0mm 0mm, width=.8\linewidth]{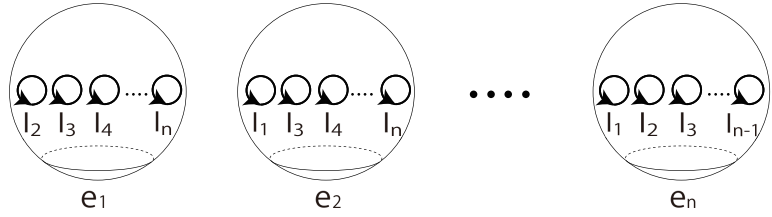}
	\end{center}
	\caption{Sectors $e_1,\ldots, e_n$ forming $X_2$}
	\label{cells_X_2}
\end{figure}

\begin{proof}
It follows by the following Lemma \ref{bipartite} that the closed surfaces obtained from all the parallel copies $e_i \times \{ \pm 1\}$ of each sector $e_i$ are connected, where $e_i \times \{ \pm 1\}$ corresponds to $v_i^{\pm}$.
This shows that any abstract dual graph $G_{X_2}$ of $X_2$ is a bouquet with $n$ loops.
Then a spanning tree $T$ of $G_{X_2}$ is a single vertex, and the $n\times n$ algebraic degree matrix $A_T$ is:
\[
A_T=\left( \begin{array}{cccccc}
0 & 1 & \cdots & \cdots & 1 \\
1 & 0 & \ddots & \ddots & \vdots \\
\vdots & \ddots & \ddots & \ddots & \vdots \\
\vdots & \ddots & \ddots & \ddots & 1 \\
1 & \cdots & \cdots & 1 & 0
\end{array} \right)
\]
Since $det(A_T)=(-1)^{n+1}(n-1)$,  we have $|det(A_T)|=n-1 \geq 2$.
Hence by Theorem \ref{main}, $X_2$ cannot be embedded into the 3-sphere.

To show the criticality of $X_2$, we construct an embedding of $X_2'$ which is obtained from $X_2$ by removing  $e_n$ (Figure \ref{X_2'}).
\begin{figure}[htbp]
	\begin{center}
	\includegraphics[trim=0mm 0mm 0mm 0mm, width=.4\linewidth]{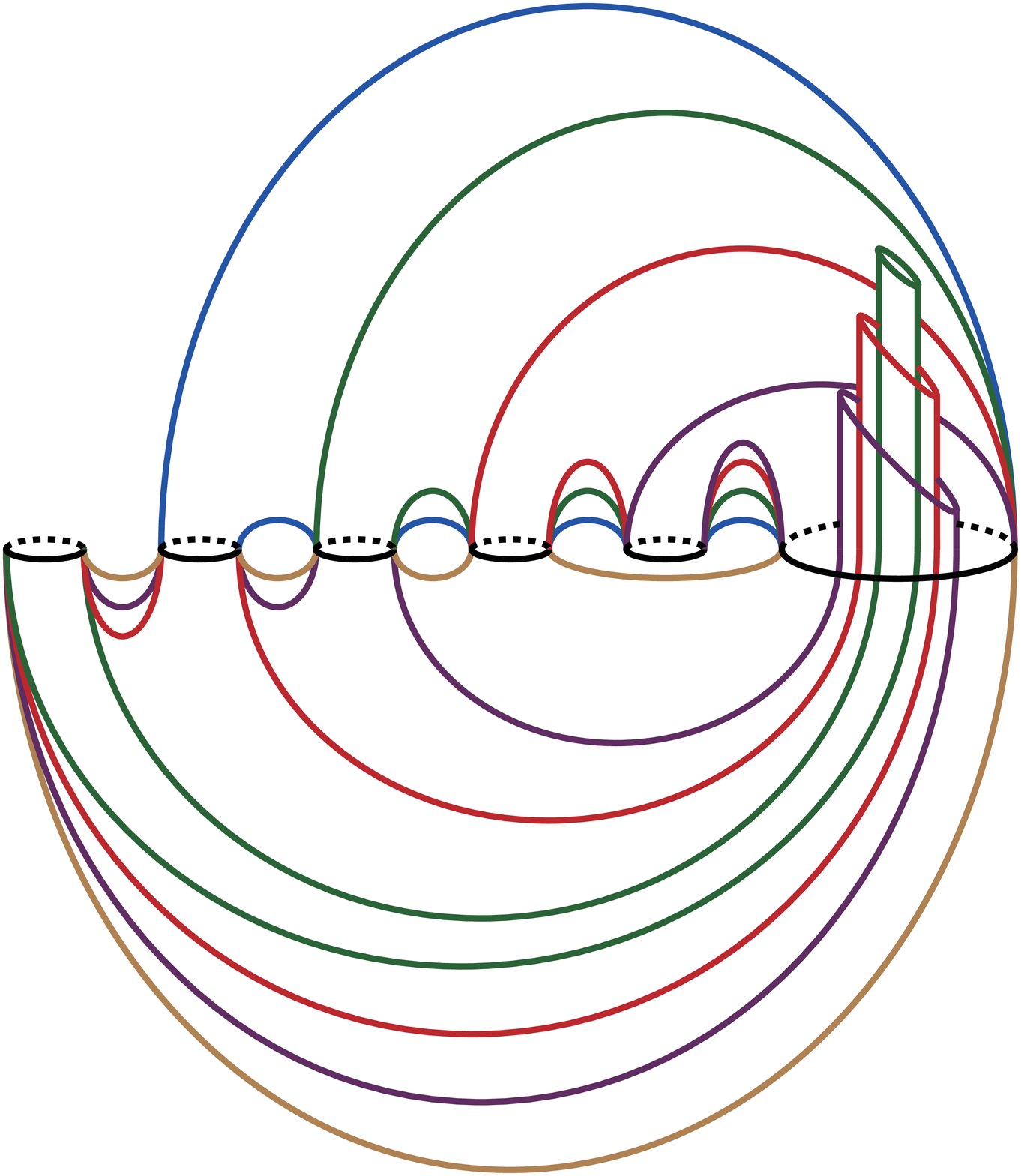}
	\end{center}
	\caption{An embedding of $X_2'$ for $n=6$}
	\label{X_2'}
\end{figure}
Next we add $e_n$ whose one point is removed to $X_2'$ (Figure \ref{cells_X_2'}).
\begin{figure}[htbp]
	\begin{center}
	\includegraphics[trim=0mm 0mm 0mm 0mm, width=.8\linewidth]{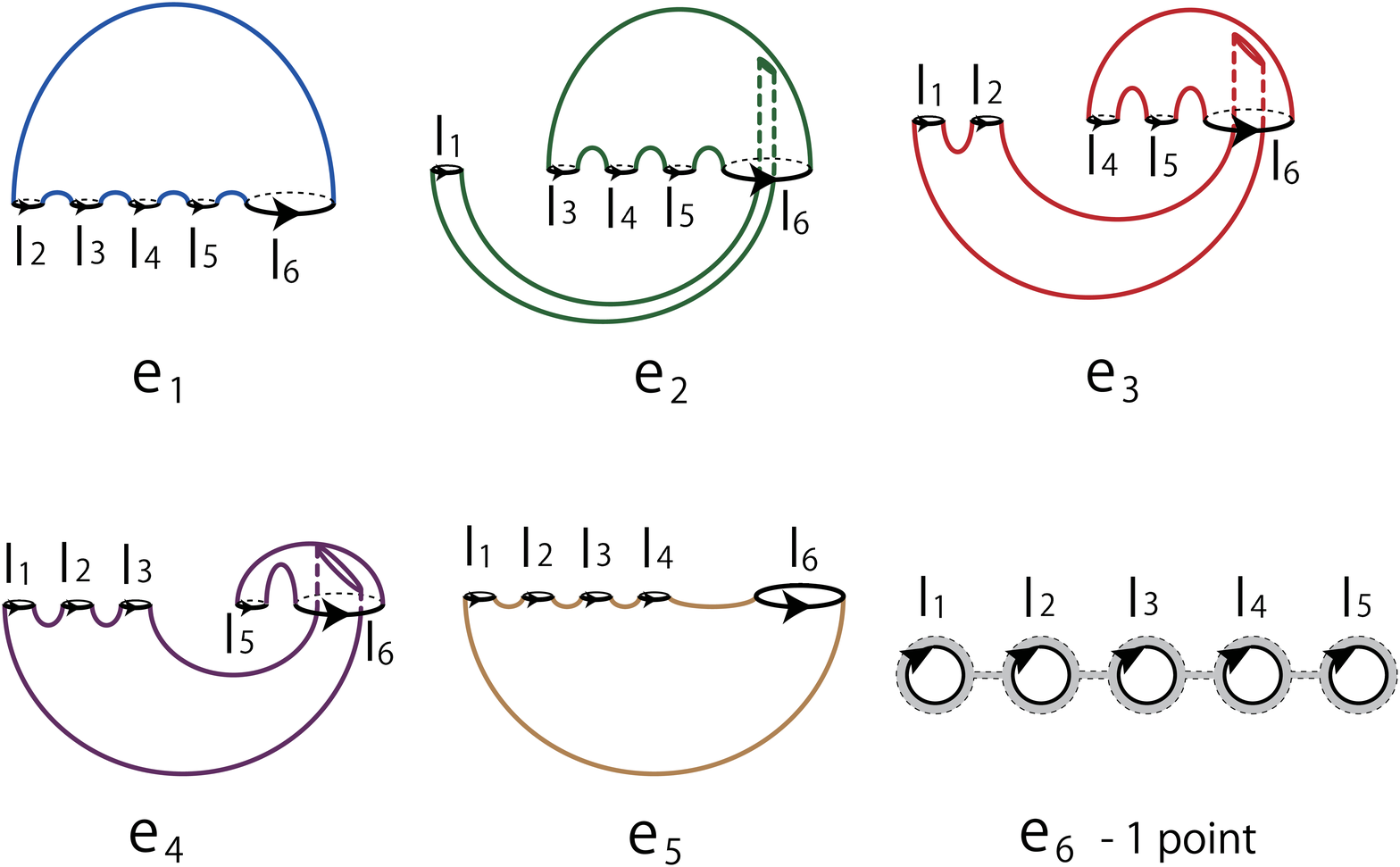}
	\end{center}
	\caption{Embeddings of each sector of $X_2$ for $n=6$}
	\label{cells_X_2'}
\end{figure}
Then we have an embedding of $X_2$ whose one point is removed into the 3-sphere.
\end{proof}

\begin{lemma}\label{bipartite}
Let $G$ be a bipartite graph with $2n$ vertices $\{v_1^+, \ldots, v_n^+\}\cup\{v_1^-, \ldots, v_n^-\}$
obtained by adding edges as follows.
For each $m$ in $\{1, \ldots, n\}$, choose a circular permutation $b_1, b_2, \dots, b_{n-1}$ of $\{1, \ldots, n\}\setminus \{m\}$, and let $\{v_{b_1}^+, v_{b_2}^-\}, \{v_{b_2}^+, v_{b_3}^-\}, \ldots, \{v_{b_{n-2}}^+, v_{b_{n-1}}^-\}, \{v_{b_{n-1}}^+, v_{b_1}^-\}$ be edges of $G$.
We add such edges for every $m$ in $\{1, \ldots, n\}$.
Then $G$ is connected.
\end{lemma}

\begin{proof}
Suppose that $G$ is disconnected, and let $V_1$ and $V_2$ be sets of vertices such that there is no edges connecting a vertex of $V_1$ with one of $V_2$.
Put $V_i^{\pm}=V_i\cap \{v_1^{\pm},\ldots, v_n^{\pm}\}$ for $i=1, 2$.

\begin{claim}\label{claim1}
$|V_1^+|=|V_1^-|$ and $|V_2^+|=|V_2^-|$.
\end{claim}
\begin{proof}
For $v_m^+\in V_2^+$, consider a circular permutation of $\{1, \ldots, n\}\setminus \{m\}$.
Then each vertex of $V_1^+$ is connected with distinct vertices of $V_1^-$ and hence $|V_1^+|\le |V_1^-|$.
The converse holds by the same argument.
\end{proof}

\begin{claim}\label{claim2}
$v_m^+\in V_1^+$ if and only if $v_m^-\in V_1^-$
\end{claim}
\begin{proof}
Suppose that $v_m^+\in V_1^+$ and $v_m^-\in V_2^-$.
Consider a circular permutation of $\{1, \ldots, n\}\setminus \{m\}$.
Then $|V_1^+ \setminus \{v_m^+\}|<|V_1^-|$ and hence there exists an edge from a vertex of $V_2^+$ to one of $V_1^-$.
This contradicts the supposition.
\end{proof}

Without loss of generality, we may assume that $v_1^+\in V_1^+$.
Considering a circular permutation of $1, \ldots, n-1$, by Claim \ref{claim2}, $v_1^+, v_1^-, v_2^+, \ldots, v_{n-1}^- \in V_1$.
Moreover, considering a circular permutation of $2, \ldots, n$, by Claim \ref{claim2}, $v_n^+, v_n^-\in V_1$.
This is a contradiction.
\end{proof}

\begin{theorem}\label{X_3}
Let $X_3$ be a multibranched surface obtained from sectors $e_1,\ldots, e_n$ by gluing along branches $l_1,\ldots, l_n$, where $n \ge 2$, $k_i \ge 1$, $k_1k_2k_3 \cdots k_n \ge 3$ as shown in Figure \ref{cells_X_3}.
Then $X_3$ is critical.
\end{theorem}

\begin{figure}[htbp]
	\begin{center}
	\includegraphics[trim=0mm 0mm 0mm 0mm, width=.8\linewidth]{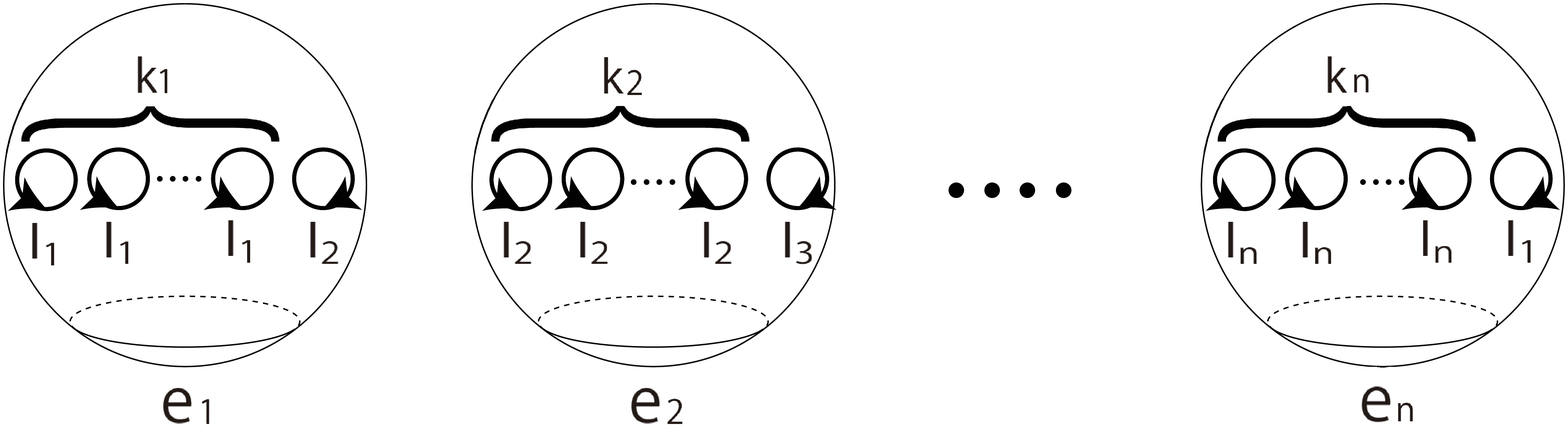}
	\end{center}
	\caption{Sectors $e_1,\ldots, e_n$ forming $X_3$}
	\label{cells_X_3}
\end{figure}

\begin{proof}
It can be observed that any abstract dual graph $G_{X_3}$ of $X_3$ is a bouquet with $n$ loops.
Then a spanning tree of $G_{X_3}$ is a single vertex, and the $n\times n$ algebraic degree matrix $A_T$ is:
\[
A_T=\left( \begin{array}{cccccc}
k_{1} & -1 & 0 & \cdots & \cdots & 0 \\
0 & k_{2} & -1 & 0 & \ddots & \vdots\\
\vdots & 0 & k_3 & -1 & \ddots & \vdots\\
\vdots & \ddots & \ddots & \ddots & \ddots & 0\\
0 & \ddots & \ddots & \ddots & k_{n-1} & -1\\
-1 & 0 & \cdots & \cdots & 0 & k_n
\end{array} \right)
\]
Since $det(A_T)=k_1k_2k_3\cdots k_n -1 \geq 2$, by Theorem \ref{main}, $X_3$ cannot be embedded into the 3-sphere.

To show the criticality of $X_3$, we construct an embedding of $X_3'$ which is obtained from $X_3$ by removing $e_n$ (Figure \ref{X_3'}).
\begin{figure}[htbp]
	\begin{center}
	\includegraphics[trim=0mm 0mm 0mm 0mm, width=.4\linewidth]{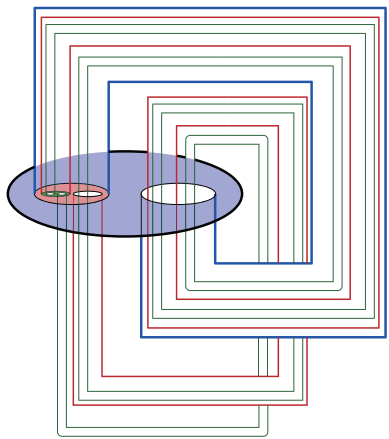}
	\end{center}
	\caption{An embedding of $X_3'$ for $n=4$}
	\label{X_3'}
\end{figure}
Next we add $e_n$ whose one point is removed to $X_3'$ (Figure \ref{cells_X_3'}).
\begin{figure}[htbp]
	\begin{center}
	\includegraphics[trim=0mm 0mm 0mm 0mm, width=.8\linewidth]{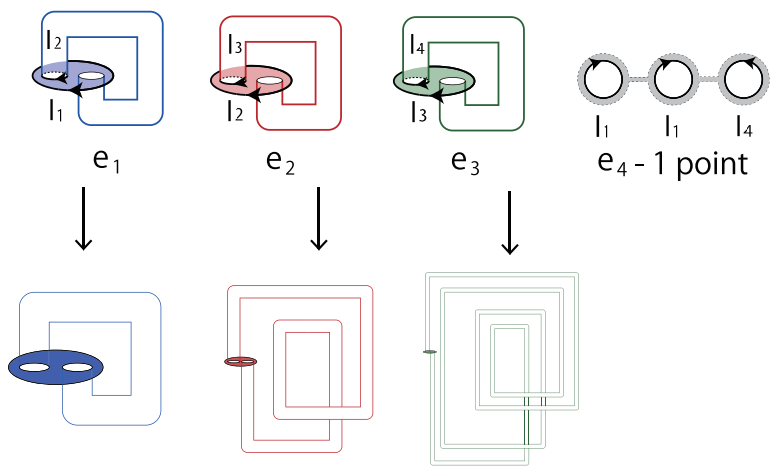}
	\end{center}
	\caption{Embeddings of each sector of $X_3$ for $n=4$}
	\label{cells_X_3'}
\end{figure}
Then we have an embedding of $X_3$ whose one point is removed into the 3-sphere.
\end{proof}

\section{Problems}

At the time of writing, we could not find an example of a connected, compact, orientable, multibranched surface which does not satisfy both conditions (1) and (2) of Theorem \ref{main}, but cannot be embedded into the 3-sphere.
Thus the following problem remains open.

\begin{problem}
Does the converse of Theorem \ref{main} hold?
\end{problem}

In the proof of Theorem \ref{X_1}, \ref{X_2} and \ref{X_3}, we have shown that any abstract dual graph of the multibranched surface is a bouquet.
It seems that this follows from the criticality of the multibranched surface, and we would raise the following problem.

\begin{problem}
Is any abstract dual graph of a critical multibranched surface a bouquet?
\end{problem}



\bibliographystyle{amsplain}

\end{document}